\documentclass[10pt]{amsart}
\usepackage{amsmath}
\usepackage{amssymb}
\usepackage{amsfonts}
\usepackage{mathrsfs}
\usepackage{graphicx}
\usepackage{epsfig}
\usepackage[T1]{fontenc}
\numberwithin{equation}{section}       

\theoremstyle{plain}
\newtheorem{theorem}{Theorem}[section]

\newtheorem{prop}{Proposition}[section]

\newtheorem{coro}[prop]{Corollary}
\newtheorem{lemma}[prop]{Lemma}

\theoremstyle{definition}
\newtheorem{definition}[prop]{Definition}

\theoremstyle{remark}
\newtheorem{remark}[prop]{Remark}

\newtheoremstyle{citing}
  {3pt}
  {3pt}
  {\itshape}
  {}
  {\bfseries}
  {.}
  {.5em}
  {\thmnote{#3}}

\theoremstyle{citing}

\DeclareMathAlphabet{\mathpzc}{OT1}{pzc}{m}{it} 

\newcommand{\C}{\mathbb{C}}

\newcommand{\Z}{\mathbb{Z}}

\newcommand{\sU}{\mathscr{U}}

\newcommand{\teta}{\widetilde{\teta}}

\newcommand{\eps}{\varepsilon}

\newcommand{\dist}{d}

\DeclareMathOperator{\diam}{diam}

\newcommand{\area}{\textrm{area}}

\begin{document}

\title[]{The Lyapunov exponent of holomorphic maps}

\author{Genadi Levin}

\address{Institute of Mathematics, The Hebrew University of Jerusalem, Givat Ram,
Jerusalem, 91904, Israel}

\email{levin@math.huji.ac.il}

\author{Feliks Przytycki}

\address{Institute of Mathematics, Polish Academy of Science, Sniadeckich St., 8, 00-956 Warsaw,
Poland}

\email{feliksp@impan.pl}

\thanks{The second author was partially supported in 2001-2014 by Polish MNiSW Grant N N201 607640. The third author is partially supported by Grant C-146-000-032-001 from National University of Singapore.}

\author{Weixiao Shen}

\address{Department of Mathematics, National University of Sinapore,
10 Lower Kent Ridge Road, Singapore 119076}
\address{
Shanghai Center for Mathematical Sciences, Fudan University, 220 Handan Road, Shanghai, China 200433
}

\email{wxshen@fudan.edu.cn}

\date{\today}

\begin{abstract}
We prove that for any polynomial map with a single critical point its
lower Lyapunov exponent at the critical value is negative if and only if
the map has an attracting cycle. Similar statement holds for the
exponential maps and some other complex dynamical systems.
We prove further that for the unicritical polynomials with positive area Julia sets
almost every point of the Julia set
has zero Lyapunov exponent. Part of this statement generalizes as follows:
every point with positive upper Lyapunov exponent in the Julia set of an arbitrary polynomial
is not a Lebegue density point.
\end{abstract}

\maketitle

\section{Introduction}

\subsection{Main results.}
In recent years, dynamical systems with different
non-uniform hyperbolicity conditions
have been studied. Speaking about one-dimensional (real or complex) dynamics,
such restrictions are often put on the
critical orbits of the map, such as Collet-Eckmann~\cite{CE}, semi-hyperbolic~\cite{CJY} and other conditions, see e.g. \cite{NS, Przytycki, GS, BvS, RS}. See also subsections~\ref{comp}-\ref{hist} of the Introduction.

Simplest and most studied are unicritical polynomial maps $f(z)=z^d+c$
and exponential maps $E(z)=a\exp(z)$. In the first two results, we
prove that for each such polynomial or exponential
map without sinks, but otherwise arbitrary,
there is always
a certain expansion along the critical orbit.

\begin{theorem}\label{uni}
Let $f(z)=z^d+c$, where $d\ge 2$ and $c\in \C$. Assume that $c$ does not belong
to the basin of an attracting cycle.
Then
$$\chi_-(f, c)=\liminf_{n\to\infty} \frac{1}{n}\log |Df^n(c)|\ge 0.$$
\end{theorem}

Theorem~\ref{uni} has been known before for real $c$
(more generally, for S-unimodal maps of an interval)~\cite{nowsan}.

Note that ~\cite{Bruin} contains examples of real quadratic polynomials $f(z)=z^2+c$ without attracting or neutral cycles
such that $\liminf_{n\to\infty} |Df^n(c)|=0$.

\begin{theorem}\label{theo:exp}
Let $E(z)=a\exp(z)$, where $a\in \C
\setminus \{0\}$. Assume that $0$ does not belong to the basin of an attracting
cycle. Then
$$\chi_-(E, 0)=\liminf_{n\to\infty} \frac{1}{n}\log |DE^n(0)|\ge 0.$$
\end{theorem}

These two theorems are special cases of the following theorem.

Let $\sU=\sU_{V,V'}$ be the set of all holomorphic maps $f:V\to V'$ between open sets $V\subset V'\subset \C$, for which there exists a unique point $c=c(f)\in V'$ and a positive number $\rho=\rho(f)$ with the following properties:

\begin{enumerate}
\item[(U1)] $f: V\setminus f^{-1}(c)\to V'\setminus \{c\}$ is an unbranched covering map;
\item[(U2)] for each $n=0,1,\ldots$, $f^n(c)$ is well defined and $B(f^n(c),\rho(f))\subset V'$.
\end{enumerate}

\begin{theorem}\label{thm:main}
For any $f\in \sU_{V,V'}$, if $c(f)$ does not belong to the basin of an attracting cycle, then $\chi_{-}(c(f))\ge 0$.
\end{theorem}

See the next Section for the proof and the last Section for some applications.
On the other hand, we have
\begin{theorem}\label{thm:typ}
Let $f(z)=z^d+c$, where $d\ge 2$ and $c\in \C$. Assume that the Julia set
$J(f)$ of $f$ has a positive area. Then for almost every $z\in J(f)$, there exists
$$\chi(f, z)=\lim_{n\to\infty} \frac{1}{n}\log |Df^n(z)|=0.$$
\end{theorem}

Quadratic polynomials with a positive area Julia set do exist~\cite{BC}.

For the proof of Theorem~\ref{thm:typ}, see Section~\ref{typ}. To show that $\chi_{-}(f, z)\ge 0$
for a typical point we introduce the notion of a {\em slowly recurrent point} $z$ for a map $f\in \sU_{V,V'}$
and prove that $\chi_{-}(f, z)\ge 0$ for such $z$, see the next Section.
In Section~\ref{typ} we show that for $f(z)=z^d+c$ a.e. point of $J(f)$ is slowly recurrent.
In the opposite direction,
that any $z\in J(f)$ with a positive upper Lyapunov exponent $\chi_{+}(f, z)$ is not a density point of $J(f)$
is an immediately consequence of the following general fact:

\begin{theorem}\label{thm:up}
Let $g$ be a polynomial of degree at least $2$.
For every $\lambda>1$ there exist $\rho>0$, a positive integer $N$ and
$\alpha>0$ as follows. Suppose
$$\chi_+(g, z)=\limsup_{n\to\infty} \frac{1}{n}\log |Dg^n(z)|>2\log \lambda$$
for some $z\in J(g)$.
Then there exists a
subset $H$ of the positive
integers such that the upper density of $H$ is at least $\alpha$,
and for every $n\in H$,
if $V_n$ denotes a connected component of $g^{-n}(B(g^{n}(z),\rho))$
which contains the point $z$ then $V_n\subset B(z, \lambda^{-n}\rho)$ and
the map $g^{n}: V_n\to B(g^{n}(z),\rho)$ is at most $N$-critical.
\end{theorem}

For the proof of a yet more general version, see Section~\ref{up}.
See also Corollary~\ref{renorm}.

Finally, for the unicritical polynomials we have the following
\begin{theorem}\label{thm:backward}
Suppose $f(z)=z^d+c$ has no an attracting cycle in $\C$.
Let $\bar x=\{x_{-n}\}_{n=0}^\infty$, $x_0=0$, $f(x_{-n})=x_{-(n-1)}$, $n>0$,
be a backward orbit of $0$. Then
$$\chi_-^{back}(f, \bar x):=\liminf_{n\to\infty} \frac{1}{n}\log |Df^n(x_{-n})|\ge 0.$$
\end{theorem}
This theorem can be deduced from Theorem~\ref{uni} by modifying the proof of Proposition 1 in \cite{GS}.
(We leave the details to an interested reader.) In Section~\ref{typ},
we shall provide a proof based on a modification of our argument in Section~\ref{1}.
\subsection{On maps with several critical values}\label{high}
There are polynomials with all periodic points repelling such that the lower Lyapunov
exponent at one of the critical values exists and is negative or even equals $-\infty$
as examples of multi-critical semi-hyperbolic, but not Collet-Eckmann polynomials show~\cite{PR0},~\cite{PRS}.
Indeed, in~\cite{PR0} a real polynomial $P$ of degree $4$ is constructed such that
all $3$ critical points of $P$  are real and non-degenerate,
two of them $q_1, q_2$ lie in the Julia set of $P$ and obey the following properties:
$P^2(q_1)$ is a repelling fixed point of $P$,
$q_2$ is non-recurrent, $P^i(q_2)\not=q_1$ for all $i\ge 0$ while $\chi_-(P, P(q_2))=-\infty$.
For a somewhat similar explicit example of degree $3$ real cubic polynomial such that the lower Lyapunov exponent at its critical value
is finite and negative, see \cite{PRS}.
\subsection{On methods of the proof}\label{comp}
As mentioned above, for S-unimodal map without periodic attractor, it was proved in~\cite{nowsan} that the lower
Lyapunov exponent of
the critical value is non-negative. The proof by Nowicki and Sands relies heavy on the order of the real line.
Using cross-ratio techniques, they found
some universal derivative and distortion bounds (related but different bounds were obtained also in \cite{Martens})
which enabled them
to compare the derivative along an orbit with the multiplier of certain periodic orbit and
thus obtain lower bounds for derivatives of
the first return maps to small neighborhoods of the critical value. To complete the proof,
they also use Ma\~n\'e's result which
asserts that an S-unimodal map without periodic attractor is uniformly expanding outside a neighborhood of the critical point.

Our
theorems cover cases where universal distortion bounds are unknown (e.g. infinitely renormalizable uncritical polynomials)
and also
cases where Ma\~n\'e's result is not applicable (e.g. maps with a Cremer periodic point).
Our approach to the lower Lyapunov exponent is
a telescope argument. Estimating the size of a ball which can be pulled back conformally and the resulting topological disk, we
derive our estimates from the (complex) Koebe distortion theorems. See Lemmas~\ref{lem:return} and~\ref{lem:awayfrom0}
in the next Section~\ref{1}.

Unlike similar arguments used frequently in
the study of expansion property of one-dimensional maps (see e.g. \cite{Nowicki, Przytycki, GS, BvS, RS}),
our argument does not make
use of any a priori knowledge of the critical orbits.
\subsection{Some motivations and historical remarks.}\label{hist}
In 1-dimensional dynamics often  the asymptotic behaviour of derivatives
along typical trajectories is reflected in the asymptotic behaviour of
derivatives along critical trajectories. E.g. hyperbolicity is equivalent to
the attraction of all critical trajectories to attracting periodic orbits.
This implies
 $\chi_-(c)< 0$ for all critical values $c$. This paper provides converse
theorems in case of a single critical point involved.

Another theory is that the `strong non-uniform hyperbolicity condition'
saying that there is
$\chi>0$ such that for all probability invariant measures $\mu$ on Julia set
$J$ for a rational map $g$\; $\chi_\mu(g):=\int \log |g'|\,d\mu\ge \chi$,  is
equivalent to so called Topological Collet-Eckmann condition, see e.g.~\cite{PRS}.
The latter in
presence of only one critical point in $J$
is equivalent to the Collet-Eckmann
condition, which says that $\chi_-(g, c)>0$ for each critical value $c\in J$ whose forward trajectory contains no critical points
(and $g$ has no parabolic orbits), see also Remark~\ref{qcinv}.

A motivation to Theorem~\ref{uni} has been the theorem saying that for all $\mu$ as
above $\chi_\mu(g)\ge 0$, see~\cite{P}.
In particular for $\mu$-almost every $x\in J$, for any $\mu$,\;
$\chi(g,x)\ge 0$.
This suggested the question whether critical values also have this property
(under appropriate assumptions).

{\bf Acknowledgment.} We thank the referee for helpful recommendations.

\section{Proof of Theorem~\ref{thm:main}}\label{1}
Let $f: V\to V'$ be a map in $\sU$ and let $c=c(f), \rho=\rho(f)$.
Furthermore,
let $AB(f)$ denote the union of the basin of attracting cycles of $f$.
So when $f$ has no attracting cycle, $AB(f)=\emptyset$.

We need two lemmas for the proof of Theorem~\ref{thm:main}.
In the first Lemma
a general construction is introduced which is used also
later on. Throughout the proofs, the Koebe principle applies.

\begin{lemma} \label{lem:return}
Assume that $c$ is not a periodic point. Given $\lambda>1$ there exists $\delta_0$, such
that for each $\delta\in (0, \delta_0)$,
if
$n\ge 1$ is the first entry time of $z\notin
AB(f)$ into $\overline{B(c;\delta)}$, then
\begin{itemize}
\item if either $|z-c|\le \delta$ or there is no
neighborhood of $z$ such that $f^n$ maps it conformally onto
$U_n=B(f^n(z), |f^n(z)-c|)$, then
\begin{equation}\label{lambdadelta}
|Df^{n}(z)|\ge \lambda^{-n}\frac{|f^n(z)-c|}{\max\{\delta, |z-c|\}};
\end{equation}
\item otherwise, i.e., if $|z-c|>\delta$ and $f^{n}$ maps a neighborhood of $z$ conformally onto
$U_n$, then
\begin{equation}\label{nolambdadelta}
|Df^{n}(z)|\ge \frac{|f^n(z)-c|}{12|z-c|}.
\end{equation}
\end{itemize}
\end{lemma}

\begin{proof} 
Let $\delta\in (0,\rho/2]$. Let $n\ge 1$ and $z\not\in AB(f)$ be as in the Lemma, and write $z_i=f^i(z)$.
Let $\{\tau_i\}_{i=0}^n$ be a sequence of positive numbers with the following properties:
\begin{enumerate}
\item $\tau_n=|z_n-c|$ and $U_n=B(z_n, \tau_n)$;
\item for each $0\le i< n$, $\tau_{i}$ be the maximal number such that
\begin{itemize}
\item $0< \tau_{i}\le \tau_{i+1}$ and
\item $f^{n-i}$ maps a neighborhood $U_i$ of $z_i$ conformally onto $B(z_n, \tau_{i})$.
\end{itemize}
\end{enumerate}


Let $$\mathcal{I}=\{0\le i<n: \tau_{i}<\tau_{i+1} \}$$ and let $$N=\# \mathcal{I}.$$
Note that for each $i\in \mathcal{I}$, $c\in \partial f(U_i)$.
Since $f^{n}$ maps $U_0$ conformally onto $B(z_n, \tau_0)$, by the Koebe $\frac{1}{4}$ Theorem, we have
\begin{equation}\label{eqn:dertau0}
|Df^{n}(z)|\ge \frac{\tau_0}{ 4\eps_0}, \text{ where } \eps_0=\dist(z_0, \partial U_0).
\end{equation}

{\bf Claim 1. } There exists a universal constant $K>1$ such that for each $i\in\mathcal{I}$, we have
$\tau_{i+1}\le K \tau_i$. Moreover,
\begin{equation}\label{eqn:diamU1}
\eps_0\le 2 \delta+|z-c|\le 3 \max\{\delta, |z-c|\}.
\end{equation}

\begin{proof}[Proof of Claim 1]
We first note that for each $0\le i<n$, $U_{i}\not\supset \overline{B(c, 2\delta)}$
for otherwise,
$$f^{n-i}(U_{i})= B(z_n, \tau_{i})\subset \overline{B(c, 2\delta)}\subset U_{i},$$
which implies by the Schwarz lemma that $z_{i}$, hence $z_0$, is contained in the basin of an attracting cycle of $f$, a contradiction!
The inequality (\ref{eqn:diamU1}) follows.

Now let $i\in \mathcal{I}$. Then $i<n-1$ and $U_{i+1}\supset \partial f(U_i)\ni c$, so
\begin{equation}\label{diam}
\diam (f(U_i))\ge |c-z_{i+1}|\ge \delta.
\end{equation}
Since $U_{i+1}\not\supset \overline{B(c,2\delta)}$,
it follows that $\textrm{mod}(U_{i+1}\setminus f(U_i))$ is bounded from above by a universal constant.
Since $f^{n-i-1}: U_{i+1}\to B(z_n, \tau_{i+1})$ is a conformal map, we have
$$\textrm{mod}(U_{i+1}\setminus f(U_i))=\log \frac{\tau_{i+1}}{\tau_i}.$$
Thus $\tau_{i+1}/\tau_i$ is bounded from above by a universal constant.
\end{proof}

By (\ref{eqn:dertau0}), it follows that
\begin{equation}\label{genin}
|Df^n(z_0)|\ge \frac{|\tau_n|}{\max\{\delta, |z-c|\}} (12 K^N)^{-1}.
\end{equation}
Since $c$ is not a periodic point,
$$C(\delta)=\inf\{m\ge 1: \exists z\in B(c,2\delta) \text{ such that } f^m(z)\in B(c,2\delta)\}\to \infty$$
as $\delta\to 0$. Thus given $\lambda>1$, there is $\delta_0>0$ such that when $\delta\in (0,\delta_0]$ we have
$$12K\le \lambda^{C(\delta)}.$$
For $i<i'$ in $\mathcal{I}\cup \{n-1\}$, we have $w:=f^{n-i'-1}(c)\in B(c;2\delta)$ and $f^{i'-i}(w)=f^{n-i-1}(c)\in B(c,2\delta)$,
$i'-i\ge C(\delta)$. Thus $n\ge C(\delta) N.$
Consider several cases. If $N\ge 1$, then
$$\lambda^n\ge \lambda^{C(\delta)N}\ge 12 K^N.$$
By~(\ref{genin}), the inequality~(\ref{lambdadelta}) holds in this case.
If $|z-c|\le \delta$,
since $z\in B(c;2\delta)$ and $f^n(z)\in B(c,2\delta)$ we have $n\ge C(\delta)$.
Hence, if $|z-c|\le \delta$ and $N=0$,
$$12\le \lambda^{C(\delta)}\le \lambda^n.$$
By~(\ref{genin}), then~(\ref{lambdadelta}) holds again.
Finally, if $N=0$, by~(\ref{genin}), the inequality~(\ref{nolambdadelta}) holds.
\end{proof}

\begin{lemma}\label{lem:awayfrom0}
There exists $M=M(f)>1$ and given $\lambda>1$ and $\delta\in (0, 1)$ there exists
$\kappa=\kappa(\delta, \lambda)$ such that whenever $z\notin AB(f)$,
$|f^j(z)-c|\ge \delta$ holds for $0< j\le n$
and $B(f^n(z),\delta)\subset V'$, we have
\begin{equation}\label{awayfrom0}
|Df^n(z)|\ge \frac{\kappa \lambda^{-n}}{\max\{M, |z-c|\}}.
\end{equation}
\end{lemma}

\begin{proof} Fix $\lambda$ and $\delta$.
We define the numbers $\tau_i$, domains $U_i$, $0\le i\le n$ and the number $\eps_0$
as in the proof of Lemma~\ref{lem:return}, with the only difference that
we start with $\tau_n=\delta$.
Let us show that there exists $M=M(f)$ such that
\begin{equation}\label{eqn:Ui2}
U_i\not\supset \overline{B(c, M+\delta)}\text{ for each }0\le i<n.
\end{equation}

If $V\not=\C$, we define
$M=\dist(c, \partial V)+1$. Then~(\ref{eqn:Ui2}) is obvious since $U_i$ must be in $V$.
If $V=\C$, then also $V'=\C$. In this case,
$f: \C\to \C$ is either a polynomial or a transcendental entire function.
It always has a (finite) periodic orbit $P$ (this fact is trivial for polynomials, and was proved by Fatou for
entire functions). Define $M=\max\{|w-c|: w\in P\}+1$.
Then~(\ref{eqn:Ui2}) holds,
for otherwise, we would have that $B(z_n, \delta)=f^{n-i}(U_i)\supset f^{n-i}(P)=P$, hence,
$B(z_n, \delta)\subset B(c, M+\delta)$.Then
$\overline{f^{n-i}(U_i)}=\overline{B(z_n, \delta)}\subset \overline{B(c, M+\delta)}\subset U_i$,
which would then imply $z_i\in AB(f)$ and hence $z\in AB(f)$, a contradiction.

It follows that
\begin{equation}\label{eqn:eps0second}
\eps_0\le M+|z-c|+\delta\le 3\max\{M, |z-c|\}.
\end{equation}

To complete the proof, we need to consider the following set of indexes
$$\mathcal{I}_\lambda=\{i\in
\mathcal{I}: \lambda \tau_i\le \tau_{i+1}\}.$$

For each $i\in\mathcal{I}_\lambda$, $\diam (f(U_i))\ge |z_{i+1}-c|\ge \delta$.
Since $f^{n-i-1}$ maps $f(U_i)$ onto $B(z_n, \tau_i)$ with a distortion which depends merely on
$\lambda$, there exists $\alpha=\alpha(\lambda)>0$ such that
\begin{equation}\label{eqn:alpha}
B(z_{i+1}, \alpha \delta) \subset f(U_i).
\end{equation}
Moreover, by (\ref{eqn:Ui2}), there exists $K=K(\delta, \lambda)>0$ such that
\begin{equation}\label{eqn:K}
\frac{\tau_{i+1}}{\tau_i}=e^{\mod(U_{i+1}\setminus f(U_i))}\le K.
\end{equation}
Furthermore, by (\ref{eqn:Ui2}) and $\tau_{i+1}\ge \lambda \tau_i$, there exists a constant $D=D(\delta,\lambda)>0$ such that
\begin{equation}\label{eqn:D}
|z_{i+1}-c|\le D.
\end{equation}

Let us prove that there exists $m_0=m_0(\lambda,\delta)$ such that $\#\mathcal{I}_\lambda\le m_0.$
To this end, let $i(0)<i(1)<\cdots<i(m-1)$ be all the elements of $\mathcal{I}_\lambda$. For each $0\le j\le j'<m$, we have
$$\mod(U_{i(j')+1}\setminus f^{i(j')-i(j)+1}(U_{i(j)}))=
\log \frac{\tau_{i(j')+1}}{\tau_{i(j)}}\ge \log\lambda\cdot (j'-j+1).$$
By (\ref{eqn:Ui2}),
it follows that there exists $m_1=m_1(\delta,\lambda)$ such that for $0\le j<j'<m$ with $j'-j\ge m_1$,
$\diam (f^{i(j')-i(j)+1}(U_{i(j)}))\le \alpha\delta/2$. For such $j, j'$, since $z_{i(j)+1}\not\in AB(f)$, $f^{i(j')-i(j)+1}(U_{i(j)})$ is not properly contained in $f(U_{i(j)})$, and thus by (\ref{eqn:alpha}), we have $|z_{i(j)+1}-z_{i(j')+1}|\ge \alpha\delta/2$.  In particular, the distance between
any two distinct points in the set
$\{z_{i(km_1)}: 0\le k< m/m_1\}$
is at least $\alpha\delta/2$. By (\ref{eqn:D}), the last set is contained in a bounded set $\overline{B(c,D)}$, thus its cardinality is bounded from above by a constant. Thus $m=\#\mathcal{I}_\lambda$ is bounded from above.


It follows that
$$\tau_0\ge \tau_n K^{-m_0}\lambda^{-(n-m_0)}\ge \delta K^{-m_0}\lambda^{-n}. $$
So
$$|Df^{n}(z)|\ge \frac{\tau_0}{4 \epsilon_0}\ge\frac{\kappa \lambda^{-n}}{\max\{M, |z-c|\}},$$
where $\kappa=\delta K^{-m_0} 4^{-1}$.
\end{proof}
\begin{remark}
Lemma~\ref{lem:awayfrom0} is valid for maps with several singular values $c_1,...c_m$. Then
the condition is that $|f^j(z)-c_i|\ge \delta$ holds for $0< j\le n$ and $1\le i\le m$ and the maximum
in~(\ref{awayfrom0}) should be taken over all $c_i$. The proof holds the same with obvious changes.
On the other hand, the proof of Lemma~\ref{lem:return} relies
on the assumption that there exists only one singular value.
\end{remark}
As a corollary of Lemmas~\ref{lem:return}-\ref{lem:awayfrom0}, we have
\begin{lemma}\label{lem:pass}
Assume that $c$ is not a periodic point.
Given $\lambda>1$ and $\sigma\in (0,1)$ there exists $C=C(\lambda, \sigma)>0$
such for every $z_0\notin AB(f)$ and $s\ge 1$ whenever
$B(f^s(z_0),\sigma)\subset V'$, we have:
if $z_0=c$, then
$$|Df^s(c)|\ge C \lambda^{-s},$$
and if $z_0\not=c$, then
$$|Df^s(z_0)|\ge C \frac{\min\{1, \inf_{i=0}^s |f^i(z_0)-c|\}}{|z_0-c|} \lambda^{-s}.$$
\end{lemma}

\begin{proof}
Fix $\lambda>1$ and $\sigma>0$, let
$\delta_0=\delta_0(\lambda)\in (0,\sigma]$ be given by Lemma~\ref{lem:return} and let
$\tilde\kappa=\kappa(\delta_0(\lambda), \lambda)$ be given by
Lemma~\ref{lem:awayfrom0}.
Define $\epsilon_0=\inf_{i=0}^s |f^i(z_0)-c|$.
If $\epsilon_0> \delta_0$, we are done by
Lemma~\ref{lem:awayfrom0} with $C=\tilde\kappa$.
Otherwise let $s_0\in \{0,...,s\}$ be minimal such that $|f^{s_0}(z_0)-c|=\eps_0$
and $s_{max}\in \{s_0,...,s\}$ be maximal such that $|f^{s_{max}}(z_0)-c|\le \delta_0$.
We show that
\begin{equation}\label{s0s}
Df^{s-s_0}(f^{s_0}(z_0))|\ge \frac{\tilde\kappa}{M} \lambda^{-(s-s_0)}.
\end{equation}
Indeed, if $s_{\max}=s_0$, we are done by Lemma~\ref{lem:awayfrom0}, where we put $z=f^{s_0}(z_0)$ and $\delta=\delta_0$.
Let $s_0<s_{max}$.
Define $\eps_1=\inf_{i=s_0+1}^{s_{max}} |f^i(z_0)-c|$ and let $s_1\in \{s_0+1,...,s_{max}\}$ be minimal
such that $|f^{s_1}(z_0)-c|=\eps_1$. If $s_1=s_{max}$ then we stop. Otherwise, define
$\eps_2=\inf_{i=s_1+1}^{s_{max}} |f^i(z_0)-c|$ and let $s_2\in \{s_1+1, \ldots, s_{max}\}$
be minimal such that $|f^{s_2}(z_0)-c|=\eps_2$. Repeating that argument,
we obtain a sequence of positive numbers $0<\eps_1\le \eps_2\le \cdots\eps_k\le \delta_0$
and a sequence of integers $0\le s_0<s_1<s_2<\cdots<s_k=s_{max}$.

Applying Lemma~\ref{lem:return} to $z=f^{s_0}(z_0)$, $\delta=\eps_1$ and $n=s_1-s_0$, we obtain
$$|Df^{s_1-s_0}(f^{s_0}(z_0))|\ge \lambda^{-(s_1-s_0)}.$$
For each $i=2,3,\ldots, k$, applying Lemma~\ref{lem:return} to $z=f^{s_{i-1}}(z_0)$, $\delta=\eps_{i}$ and $n=s_i-s_{i-1}$, we obtain
$$|Df^{s_i-s_{i-1}}(f^{s_{i-1}}(z_0))|\ge \lambda^{-(s_i-s_{i-1})}.$$
Therefore
$$|Df^{s_{max}-s_0}(f^{s_0}(z_0))|=\prod_{i=1}^k |Df^{s_i-s_{i-1}}(f^{s_{i-1}}(f^{s_0}(z_0)))|\ge \lambda^{-(s_{max}-s_0)}.$$
If $s_{max}<s$, we can further apply Lemma~\ref{lem:awayfrom0} with $z=f^{s_{max}}(z_0)$, $\delta=\delta_0$ and $n=s-s_{max}$:
$$|Df^{s-s_{max}}(f^{s_{max}}(z_0))|\ge \frac{\tilde\kappa}{M} \lambda^{-(s-s_{max})}.$$
Thus~(\ref{s0s}) follows:
$$|Df^{s-s_0}(f^{s_0}(z_0))|\ge \lambda^{-(s_{max}-s_0)}\frac{\tilde\kappa}{M} \lambda^{-(s-s_{max})}=
\frac{\tilde\kappa}{M}\lambda^{-(s-s_0)}.$$
This proves the Lemma if $s_0=0$ (including the case $z_0=c$).

If $s_0>0$, we apply Lemma~\ref{lem:return} for $z=z_0$, $\delta=\epsilon_0$
and $n=s_0$:
$$|Df^{s_0}(z_0)|\ge \lambda^{-s_0} \frac{\epsilon_0}{12 \max\{\epsilon_0, |z_0-c|\}}=\lambda^{-s_0}
\frac{\epsilon_0}{12 |z_0-c|}.$$
Combining this with~(\ref{s0s}) we obtain finally:
$$|Df^s(z_0)|\ge C \frac{\inf_{i=0}^s |f^i(z_0)-c|}{|z_0-c|} \lambda^{-s},$$
where $C=\tilde\kappa/(12M)$.
\end{proof}

Theorem~\ref{thm:main} follows at once from Lemma~\ref{lem:pass}.
\begin{proof}[Proof of Theorem~\ref{thm:main}]
We may certainly assume that $c$ is not periodic. Fix $\lambda>1$.
Applying Lemma~\ref{lem:pass} with $z_0=c$, we find $C>0$ such that for each $s\ge 1$,
$$|Df^s(c)|\ge C \lambda^{-s}.$$
Hence, $\chi_{-}(f, c)\ge -\log\lambda$ for every $\lambda>1$
\end{proof}

Let us formulate another consequence of Lemma~\ref{lem:pass}.

\begin{definition} Let $z\notin AB(f)$ such that the forward orbit of $z$ is well-defined.
We call $z$
{\em (exponentially) slowly recurrent} if
for any $\alpha>0$, $|f^n(z)-c|\ge e^{-\alpha n}$ holds for every large $n$.
\end{definition}

As in the proof of Theorem~\ref{thm:main}, Lemma~\ref{lem:pass} implies:
\begin{lemma}\label{lem:slexp}
If $z$ is a slowly recurrent point and there exists $\delta>0$
such that $B(f^n(z), \delta)\subset V'$ for every $n$,
then $\chi_{-}(f,z)\ge 0$.
\end{lemma}
\section{Expansion along the orbit}\label{up}
Let us fix a polynomial $g$ of degree at least $2$.
Theorem~\ref{thm:up} of the Introducton is an immediate corollary
of the following.
\begin{theorem}\label{thm:positiv}
For every $\lambda>1$ and $\epsilon_0>0$ there exist $\rho>0$,
$N, \tilde n\in \Z^+$ and
$\alpha>0$ (which depend on $g$ and $\lambda$, $\epsilon_0$ only) as follows.
Let $z\in J(g)$ and $m\in \Z^+$. Assume
\begin{equation}\label{chi+}
\frac{1}{m}\log |Dg^m(z)|>\epsilon_0+\log \lambda.
\end{equation}
Then there exists a
subset $H_m$ of the set $\{1,2,...,m\}$ such that the following hold:
\begin{enumerate}
\item[(a)] $\#\ H_m\ge \alpha m$,
\item[(b)] for every $n\in H_m$,
if $V_n$ is a connected component of $g^{-n}(B(g^{n}(z),\rho))$
which contains the point $z$ then
the map $g^{n}: V_n\to B(g^{n}(z),\rho)$ is at most $N$-critical, that is,
$\#\{0\le k < n:
g^k(V_n)\cap Crit \not= \emptyset \} \le N$,
where $Crit$ is the set of critical points of $g$.

\item[(c)] $V_n\subset B(z, \lambda^{-n}\rho)$ whenever $n\in H_m$ and $n\ge \tilde n$.
\end{enumerate}
\end{theorem}

This result has already been referred to in~\cite{GPR}.
\begin{remark}\label{ratexp}
Theorem~\ref{thm:positiv} as well as its proof (see below) hold for rational functions as well
(derivatives are then
taken in the spherical distance).
\end{remark}

We need some preparations for the proof.
\begin{definition} We say that $m$ is a {\em $\lambda$-hyperbolic time} for $z$
if $$|Dg^{m-i}(g^i(z))|\ge \lambda^{m-i}$$ holds for each $i=0,1,\ldots, m-1$.
\end{definition}
Let $W_\delta^k(z)$ , $k>0$, denote the connected component of $g^{-k}(B(g^k(z),\delta))$ which contains $z$.
The next Lemma is known as the ``telescope lemma'' of~\cite{P0} (see also~\cite{PRS}, Lemma 2.3
for a simplified method and~\cite{GPR} for $C^1$ multimodal interval maps):
\begin{lemma}\label{lem:hyp}
Let $\lambda>1$ and $\epsilon>0$.
There exist $C=C(\lambda,\epsilon)$ and
$\delta_0=\delta_0(\lambda,\epsilon)>0$ as follows.
Assume that $s$ is a
$\lambda e^{\epsilon}$-hyperbolic time for $z\in J(g)$. Then, for
every $\delta\in (0, \delta_0]$ and every $i=0,...,s-1$,
\begin{equation}\label{backall}
W_\delta^{s-i}(g^i(z))\subset B(g^i(z), C \delta \lambda^{-(s-i)}).
\end{equation}
\end{lemma}

\

Recall the definition of ``shadow''~\cite{PR}.
Fix $z\in J(g)$. For $n\in \Z^+$, set $\varphi(n)=-\log d(g^n(z), Crit)$.
Multiplying the metric $d$ by a constant we can assume that $\varphi\ge 0$.
By~\cite{DPU}, there exists $C_g>0$, which depends only on $g$
such that
\begin{equation}\label{psibound}
\sum_{k=0}^{n-1}{\bf '} \ \varphi(k)\le C_g n, \ \ \ n=1,2,...
\end{equation}
where $\sum {\bf '}$ denotes the summation over all but at most $M=\#\ Crit$ indexes
(at most one closest approach, i.e., the biggest $\varphi(k)$, per each $c\in Crit$).
Given $K>0$, define a ``shadow'' $S(j, K)$ of
$j\in \Z^+$ to be the following interval of the real line:
$$S(j, K)=(j, j+K \varphi(j)].$$
For any $N\in \Z^+$, let $A(N, K)$ be the set of all
$n\in \Z^+$ such that there are at most $N$ integers $j$ so that
$n\in S(j, K)$.
Let us denote by $l(S)$ the lenght of an interval $S\subset {\bf R}$.

{\bf Claim 1.}
For every $m\in \Z^+$,
$$\frac{\#\ \{A(N, K)\cap \{1,..., m\}\}}{m}\ge 1-\frac{C_g K}{N-M+1}.$$

\begin{proof}[Proof of Claim 1]
Let $1\le n_1<...<n_r\le m$ be all $n$ with the property that
$n$ lies in at least $N+1$ ``shadows''. Then obviously
$\sum_{j=0}^{m-1}{\bf '} \ l(S(j, K))\ge (N-M+1) r$.
On the other hand, by~(\ref{psibound}),
$${\sum_{j=0}^{m-1}}{\bf '} \ l(S(j, K))=K \sum_{j=0}^{m-1}{\bf '} \
\varphi(j)\le K C_g m.$$
Therefore, for $N$ large enough, $r\le \frac{C_g K}{N-M+1} m$.
\end{proof}

Given $r>0$, we denote by $G(N, r)$ the set of all
$n\in {\bf N}$ such that the map
$g^n: W_r^n(z)\to B(g^n(z), r)$ is at most $N$-critical.
Then we have

{\bf Claim 2.}
Suppose that $n$ is an $\lambda e^\epsilon$-hyperbolic time for $z$, for some $\epsilon>0$.
Set $K_0=1/\log\lambda$. There  exists $\tilde\delta=\tilde\delta (\lambda, \epsilon)$
such that for every $\delta\in (0, \tilde\delta]$ and $N\in \Z^+$,
if $n\in A(N, K_0)$, then $n\in G(N, \delta)$.
\begin{proof}[Proof of Claim 2.]
Let $C=C(\lambda, \epsilon)$ and $\delta_0=\delta_0(\lambda, \epsilon)$
be taken from Lemma~\ref{lem:hyp}. Define $\tilde\delta=\min\{\delta_0, 1/C\}$.
Let $\delta\le \tilde\delta$.
If $1\le j<n$ is such that $W_\delta^j(g^{n-j}(z))\cap Crit\not=\emptyset$, then
$d(g^{n-j}(z), Crit)\le \diam(W_\delta^j(g^{n-j}(z)))\le C \delta \lambda^{-j}$.
Hence,
$$\varphi(n-j)\ge j\log\lambda + \log\frac{1}{C \delta}\ge
\frac{j}{K_0}.$$
It means that $n\in S(n-j, K_0)$.
Now, assume that $g^n: W_\delta^n(z)\to B(g^n(z), \delta)$ is {\em at least}
$N+1$-critical.
By the preceding consideration,
$n$ belongs to at least $N+1$ different
``shadows'' $S(n-j, K_0)$. Therefore, $n\notin A(N, K_0)$.
Thus $n\in A(N, K_0)$ implies $n\in G(N, \delta)$.
\end{proof}

The last ingredient of the proof is the Pliss Lemma~\cite{Pliss}:
\begin{lemma}\label{pliss}
Let $0<b_1<b_2\le B$ and $\theta=(b_2-b_1)/(B-b_1)$. Given real numbers
$a_1,...,a_r$ satisfying $\sum_{j=1}^r a_j\ge b_2 r$ and $a_j\le B$ for all
$1\le j\le r$, there are $l>\theta r$ and $1<n_1<...<n_l\le r$ such that
$\sum_{i=n+1}^{n_j} a_i\ge b_1(n_j-n)$ for each $0\le n<n_j$, $j=1,...,l$.
\end{lemma}

\begin{proof}[Proof of Theorem~\ref{thm:positiv}]
Let $z\in J(g)$, $m\in \Z^+$ and
$$\frac{1}{m}\log|Dg^{m}(z)|>\epsilon_0+ \log\lambda.$$
Let $\epsilon=\epsilon_0/8$. Consider the set
$$\tilde H_m=\{n\in \{1,...,m\}: n \mbox{ is a }
\lambda e^{4\epsilon} -\mbox{hyperbolic time for }
z\}.$$
By the Pliss Lemma~\ref{pliss} (for $a_j=\log
g'(g^{j-1}(z))|$), there exists $\theta=\theta(\lambda, \epsilon)$
such that
\begin{equation}\label{hmk}
\frac{\#\ \tilde H_m}{m}>\theta, \ m\ge 1.
\end{equation}
Let $K_0=1/\log(\lambda e^{2\epsilon})$ and let
$\tilde\delta=\tilde\delta(\lambda, 2\epsilon)$ be the constant from Claim 2.
We define
$$\rho=\tilde\delta, \ N=[2 C_f K_0/\theta]+M \mbox{ and }
\alpha=\theta/2.$$
Consider the corresponding sets $A(N, K_0)$ and $G(N, \rho)$.
By Claim 1,
\begin{equation}\label{cl1}
\frac{\#\ \{A(N, K_0)\cap \{1,..., m\}\}}{m}\ge 1-\frac{C_g K_0}{N-M+1}>1-\frac{\theta}{2}
\end{equation}
while by Claim 2,
\begin{equation}\label{cl2}
A(N, K_0)\cap \tilde H_m\subset G(N, \rho).
\end{equation}
Then~(\ref{hmk}) and~(\ref{cl1}) give us
\begin{equation}\label{cl3}
\#\ \{A(N, K_0)\cap \tilde H_m\}>\frac{\theta}{2} m.
\end{equation}
Hence, by~(\ref{cl2}) we must have:
\begin{equation}\label{g}
\frac{\#\ \{G(N, \rho)\cap \tilde H_m\}}{m}>\alpha.
\end{equation}
Denote $H_m=G(N, \rho)\cap \tilde H_m$.
Then the properties (a) and (b)
hold for $n\in H_m$. Furthermore, by Lemma~\ref{lem:hyp},
for every $n\in H$, $\diam V_n\le 2 C (\lambda e^{2\epsilon})^{-n}$, where $C=C(\lambda, 2\epsilon)>0$.
Hence, $\diam V_n\le \lambda^{-n}$, for every $n\ge \tilde n$,
where $\tilde n=\tilde n(\lambda, \epsilon)$. Thus (c) holds too.
\end{proof}

\begin{remark}\label{notdens}
Theorem~\ref{thm:positiv}
implies the following fact which is announced in the Abstract: every point $z$ of a polynomial Julia set $J(g)$ satisfying
$\chi_{+}(g, z)>0$ is not a Lebesgue density point of $J(g)$.
(By Remark~\ref{ratexp}, this holds as well
for rational functions with nowhere dense Julia sets.)
In fact $J(g)$ is even {\it upper mean porous} at such $z$
which means:
there are $r>0$ and a
subset of $\Z^+$ of a positive  upper density such that
for every $j$ in this subset $B(z,2^{-j})$
contains a ball disjoint from $J(g)$ of radius $r2^{-j}$.
The proof is the same as in~\cite{PR} in the case of $g$ satisfying
Collet-Eckmann condition.
\end{remark}

\section{Unicritical polynomials}\label{typ}
In this Section, $f(z)=z^d+c$, where $d\ge 2$ and $c\in\C$.
Recall that a point $z\in J(f)$ is {\em slowly recurrent}
if for any $\alpha>0$, $|f^n(z)|\ge e^{-\alpha n}$ holds for every large $n$.

The next fact is crucial.
\begin{lemma}\label{lem:slarea}
Almost every point $z\in J(f)$ is slowly recurrent.
\end{lemma}

\begin{proof}
If the critical point $0$ is not in the Julia set, then $0$ is attracted by either an attracting cycle or a parabolic cycle, and it is well known that the Julia set has measure zero. In the following, we assume that $0\in J(f)$.

We need the following fact which is a particular case of~\cite{P}, Lemma 1.

{\bf Claim 1.} There is a constant $K=K(d)>0$ such that for any $\eps>0$, any $w\in J(f)$
and any integer $s\ge 1$, if $|w|<\eps$ and $|f^s(w)|<\eps$,
then $$s\ge C(\eps):= K\log(\eps^{-1}).$$

Now we fix $\alpha>0$ 
and consider $E_n=\{z\in J(f): |f^n(z)|<e^{-2\alpha n}\}$. Let $V$ be a component
of $f^{-n}(B(0, e^{-\alpha n}))$.
If $0\le j_1<j_2\le n$ are such that
$f^{j_1}(V)$ and $f^{j_2}(V)$ both contain $0$, then $j_2-j_1$ is a return time of
$f^{n-j_2}(0)\in B(0, e^{-\alpha n})$ into the ball. By Claim 1, $j_2-j_1\ge C(e^{-\alpha n})=\tilde C\alpha n$
for some constant $\tilde C=\tilde C(d)>0$. Thus
$\#\ \{0\le j\le n: f^j(V)\ni 0\}$ is bounded by
$1+1/(\alpha C)$. It means that the map
$f^n: V\to B(0, e^{-\alpha n})$ is a (branched) cover with a degree which is
uniformly bounded by some $D=D(d, \alpha)$.
Since $E_n\cap V\subset f^{-n}(B(0, e^{-2\alpha n}))$, it follows
from a version of the Koebe distortion theorem for multivalent maps, see e.g.~\cite{PR}, Lemma 2.1,
that there exists $\alpha'>0$, which depends only on $\alpha$ and $D$
such that $\area(V\cap E_n)/\area(V)\le e^{-\alpha' n}$.
As $V\subset f^{-n}(B(0, e^{-\alpha n}))\subset B(0, 3)$,
$\area(E_n)$ is exponentially small with $n$ and thus $\sum_{n=1}^\infty \area(E_n)<\infty$.
By the Borel-Cantelli lemma, a.e. $z$ is
contained in only finitely many $E_n$.
\end{proof}

\begin{proof}[Proof of Theorem~\ref{thm:typ}]
By Lemma~\ref{lem:slexp}, $\chi_{-}(z)\ge 0$ for every slowly recurrent $z\in J(f)$ and using
Lemma~\ref{lem:slarea}, $\chi_{-}(z)\ge 0$ for almost every $z\in J(f)$.
And $\chi_{+}(z)\le 0$ for a.e. $z$ because each $z$ with
$\chi_{+}(z)> 0$ is not a Lebesgue density point by Theorem~\ref{thm:up}.
\end{proof}

We shall now consider backward orbits of the critical point and prove Theorem~\ref{thm:backward}
of the Introduction.
We first prove the following variation of Lemma~\ref{lem:return} for polynomial maps
which gives a better estimate.

\begin{lemma} \label{lem:returnpoly}
Let $f(z)=z^d+c$ where $d\ge 2$ and $c\in\C$. Assume that $f$ has no an attracting cycle. Then for each $\lambda>1$ there is $\delta_0>0$ such that for each $\delta\in (0,\delta_0)$, if $z\not\in B(0,\delta)$ and $n\ge 1$ is the minimal positive integer such that $|f^n(z)|\le \delta$, then
$$|Df^n(z)|\ge \frac{\delta}{12|z|}\lambda^{-n}.$$
Moreover, if $|z|=\delta$, then $$|Df^n(z)|\ge \lambda^{-n}.$$
\end{lemma}
\begin{proof} Fix $\lambda$ and $\delta$. Define the numbers $\tau_i$, domains $U_i$, $0\le i\le n$, the index set $\mathcal{I}$ and the number $\eps_0$ as in the proof of Lemma~\ref{lem:return}, with the only difference that we start with $U_n=B(z_n,\delta)$, so $\tau_n=\delta$. As $f$ has no attracting cycle in $\C$, $U_i\not\supset \overline{B(0,2\delta)}$ for each $0\le i\le n$.  Thus the conclusion of Claim 1 of Section~\ref{1}
holds with $c$ replaced by $0$ in the current setting. So we obtain similar to (\ref{genin}) the following estimate:
$|Df^n(z)|\ge \delta/(12|z| K^{N}).$
Arguing as in the last part of the proof of Lemma~\ref{lem:return}, we show that $K^{N}\le \lambda^{n}$,  provided that $\delta$ is small enough.
In the case $|z|=\delta$, we have $12 K^{N}\le \lambda^{n}.$ Thus the lemma holds.
\end{proof}

As an immediate corollary, we have
\begin{lemma}\label{lem:closereturn}
Let $f(z)=z^d+c$ where $d\ge 2$ and $c\in\C$. Assume that $f$ has no an attracting cycle. Then for each $\lambda>1$ there is $\delta_0>0$ such that for each $z\in\C$ and $n\ge 1$ with $|f^n(z)|\le \delta_0$ and with $|f^n(z)|\le |f^j(z)|$ for each $0\le j<n$,  we have
$$|Df^n(z)|\ge \min\left(\frac{\delta_0}{12|z|},1\right) \lambda^{-n}.$$
\end{lemma}
\begin{proof}
Let $\delta_0$ be given by Lemma~\ref{lem:returnpoly} and let $\delta_0'=\min(|z|,\delta_0)$.
By the assumption, there is a sequence of integers $1\le n_1<n_2<\cdots<n_k=n$ such that
\begin{itemize}
\item $n_1$ is the minimal positive integer such that $|f^{n_1}(z)|\le \delta_0'$;
\item $n_{j+1}$ is the minimal integer with $n_{j+1}>n_j$ and $|f^{n_{j+1}}(z)|\le |f^{n_j}(z)|$ for each $j=1,\ldots, k-1$.
\end{itemize}
If $|z|\ge \delta_0$, then $n_1$ is the minimal positive integer such that $|f^{n_1}(z)|\le \delta_0$, so
by Lemma~\ref{lem:returnpoly}, we have
$$|Df^{n_1}(z)|\ge \frac{\delta_0}{12|z|}\lambda^{-n}.$$
If $|z|<\delta_0$, then $|z|=\delta_0'\in (0,\delta_0]$ and $n_1$ is the minimal integer such that $|f^{n_1}(z)|\le \delta_0'$, so by the latter inequality of Lemma~\ref{lem:returnpoly}, we have
$$|Df^{n_1}(z)|\ge \lambda^{-n_1}.$$
For each $1\le j<n$, putting $\delta_j:=|f^{n_j}(z)|\in (0,\delta_0]$, $n_{j+1}-n_j$ is the minimal positive integer such that |$f^{n_{j+1}-n_j}(f^{n_j}(z))|\le \delta_j$, so by the latter inequality in Lemma~\ref{lem:returnpoly} again, we have
$$|Df^{n_{j+1}-n_j}(f^{n_j}(z))|\ge \lambda^{-(n_{j+1}-n_j)}.$$
Therefore, if $|z|\ge \delta_0$, then
$$|Df^n(z)|=|Df^{n_1}(z)|\prod_{j=1}^{k-1} |Df^{n_{j+1}-n_j}(f^{n_j}(z))|\ge \frac{\delta_0}{12|z|}\lambda^{-n},$$
and if $|z|<\delta_0$, then
$$|Df^n(z)|\ge \lambda^{-n}.$$
\end{proof}

\begin{proof}[Proof of Theorem~\ref{thm:backward}]
Fix $\lambda>1$ and let $\delta_0$ be given by Lemma~\ref{lem:closereturn}. Let $\bar{x}=\{x_{-n}\}_{n=0}^\infty$ be a backward orbit of $0$. Then for each $n$, applying Lemma~\ref{lem:closereturn} to $z=x_{-n}$, we obtain
$$|Df^n(x_{-n})|\ge K_n \lambda^{-n},$$
where $K_n=\min \left(\delta_0/(12|x_{-n}|), 1\right).$ Clearly $K_{n}$ is bounded from below by a positive constant depending only on $f$. Thus
$$\chi_-^{back}(f, \bar x) \ge -\log \lambda.$$
Since this holds for all $\lambda>1$, $\chi_-^{back}(f,\bar x)\ge 0$.
\end{proof}

\section{Some applications and remarks}
Let us note the following special case of Theorem~\ref{thm:main}:
\begin{theorem}\label{rat}
Let $g$ be a rational function on the Riemann sphere of degree at least two.
Given a critical value $c$ of $g$,
define its postcritical
set $P(c)=\overline{\cup_{n\ge 0}g^n(c)}$.
Assume that $c_0$ is a critical value of $g$ not in the basin of
an attracting cycle, such that $P(c_0)$ is disjoint
from the union $X$ of the postcritical sets of
all other critical values of $g$.
Then
$$\chi_-(g, c_0)=\liminf_{n\to\infty} \frac{1}{n}\log \|Dg^n(c_0)\|\ge 0,$$
where $\|\cdot\|$ denote the norm in the spherical metric.
\end{theorem}

\begin{proof}
By means of a M\"obius conjugacy, we may assume that $\infty\in X$, so that the orbit of $c_0$
lies in a compact subset of $\C$ and
$\chi_-(g,c_0)$ can be calculated using the Euclidean metric instead of the spherical metric.
Then define $V'=\overline{\C}\setminus X$ and $V=g^{-1}(V')$, and apply Theorem~\ref{thm:main} to $g\in \sU_{V, V'}$.
\end{proof}

An immediate corollary of Theorem~\ref{uni} along with
Remark 13 of~\cite{lanal} is as follows:
\begin{coro}\label{ruelle}
Assume that the map $f(z)=z^d+c$ has no attracting cycles.
Then the power series
$$F(t)=1+\sum_{n=1}^\infty\frac{t^n}{Df^n(c)}$$
has the radius of convergence at least $1$, and
\begin{equation}\label{fr}
F(t)\not=0 \text{ for every } |t|<1.
\end{equation}
\end{coro}
\begin{remark}The function $F(t)$
should be interpreted as ``Fredholm determinant''
of the operator $T: \phi\mapsto \sum_{w: f(w)=z}\frac{\phi(w)}{Df(w)^2}$
acting in a space of functions $\phi$, which are analytic outside of
$J(f)$ and locally integrable on the plane.
Then~(\ref{fr}) reflects the fact that $T$ is
a contraction operator in this space.
Note that this operator plays, in particular, an important
role (after Thurston)
in the problem of stability in holomorphic dynamics.
\end{remark}

Another consequence of Theorem~\ref{thm:main} is that:
\begin{coro}\label{cor:pollike}
Let $g_i: V_i\to V_i'$, $i=0,1$, be two mappings
in the class $\sU$ which are quasi-conformally conjugated
(i.e., there exists a q-c map $h: \C\to \C$ such that $h(V_0)=V_1$,  $h(V_0')=V_1'$, and $h\circ g_0=g_1\circ h$ on $V_0$).
Assume that $\omega_{g_0}(c(g_0))$ (the $\omega$-limit set of the point $c(g_0)$ by the map
$g_0$) is compactly contained in $V_0$. If, for a subsequence $n_k\to \infty$,
$\lim_{k\to\infty} \frac{1}{n_k}\log |Dg_0^{n_k}(c(g_0))|=0$,
then also $\lim_{k\to\infty} \frac{1}{n_k}\log |Dg_1^{n_k}(c(g_1))|=0$.
\end{coro}

\begin{proof}
Normalize the maps in such a way that $c(g_0)=c(g_1)=0$ and $h(1)=1$.
As in~\cite{mss}, one can include $g_0$ and $g_1$ in a family $g_\nu$ of quasi-conformally conjugated maps of the class
$\sU$, with $c(g_\nu)=0$,
which depends holomorphically on
$\nu\in D_r=\{|\nu|<r\}$, for some $r>1$. Namely, if $\mu=\frac{\partial h}{\partial \bar z}/\frac{\partial h}{\partial z}$ is complex dilatation
of $h$, then, for every $\nu\in D_r$, where $r=||\mu||_{\infty}^{-1}$, let $h_\nu$ be the unique q-c homeomorphism of $\C$ with complex dilatation
$\nu \mu$, which leaves the points $0,1$ fixed (in particular, $h_0=id$ and $h_1=h$). Then we can define domains $V_\nu=h_\nu(V_0)$,
$V_\nu'=h_\nu(V_0')$, and the map
$g_\nu=h_\nu\circ g_0\circ h_\nu^{-1}: V_\nu\to V_\nu'\in \sU$ with $c(g_\nu)=0$.
As $\omega_{g_\nu}(0)=h_\nu(\omega_{g_0}(c(g_0))$ is compactly contained in $V_\nu=h_\nu(V_0)$,
by the Schwarz lemma and a compactness argument,
given a compact subset $K$ of the disk $D_r$, there exists $C$, such that
$|Dg_\nu(g_\nu^i(0))|\le C$ for every $\nu\in K$ and every $i\ge0$.
Then $u_k(\nu)=n_k^{-1}\log |Dg_\nu^{n_k}(0)|$ is a sequence
of harmonic functions in $D_r$, which is bounded on compacts. On the other hand, by Theorem~\ref{thm:main},
every limit value of the sequence $\{u_k\}$ is non-negative, and, by the assumption,
$u_k(0)\to 0$.
According to the Minimum Principle, $u_k(\nu)\to 0$ for any $\nu$.
\end{proof}

\begin{remark}\label{qcinv}
In particular, the Collet-Eckmann condition
$\chi_-(g, c(g))>0$ is a quasi-conformal invariant.
In fact, it is even a topological invariant~\cite{PR0}.
\end{remark}

%

The following is a consequence of Theorem~\ref{thm:up}:
\begin{coro}\label{renorm}
Let $g$ be a polynomial which is infinitely-renormalizable around a critical
value $c$.
If $\chi_{+}(g, c)>0$, then
$J(g)$ is not locally-connected.
\end{coro}

Indeed,
let $J_n$, $n\ge 1$, be a sequence of ``small'' Julia sets such that $c\in J_n$ and let $p_n\to \infty$
be their corresponding periods. Theorem~\ref{thm:up} implies that there is a sequence of integers
$m_n\to \infty$ such that $\inf_{n} \diam g^{m_n}(J_n)>0$. Hence, a Hausdorff limit point of the sequence
of compacts $g^{m_n}(J_n)$, $n\ge 1$, must be a  non-trivial wandering subcontinuum of $J(g)$.
By Theorem 3.2 of~\cite{bl} (see also~\cite{lback}),
$J(g)$ cannot be locally-connected.

\end{document}